\documentclass[10pt,a4paper, xcolor]{article}

\usepackage{amsfonts}
\usepackage{comment}
\usepackage{amssymb,amsthm, amsmath,graphicx,bm,ebezier,amsthm,mathtools}
\usepackage{color,enumerate}
\usepackage{hyperref}
\usepackage{url}
\usepackage{color,enumerate}
\usepackage[utf8]{inputenc}
\usepackage{algorithm}
\usepackage[noend]{algpseudocode}
\usepackage{tikz}
\usepackage{qtree}

\newtheorem{theorem}{Theorem}
\newtheorem{definition}[theorem]{Definition}
\newtheorem{proposition}[theorem]{Proposition}

\newtheorem{lemma}[theorem]{Lemma}
\theoremstyle{remark}

\newtheorem{example}[theorem]{Example}
\newtheorem{remark}[theorem]{Remark}

\newcommand{\W}{\textsf W}

\renewcommand{\v}{\overrightarrow}
\renewcommand{\L}{{\tt{L}}}
\newcommand{\LL}{{\mathcal L}}

\newcommand{\Q}{{\mathbf Q}}
\newcommand{\N}{{\mathbf N}}

\newcommand{\Z}{{\mathbf Z}}

\renewcommand{\mod}{{\rm mod}}

\def\z{{\sf Z}}

\def\lcm{\mathrm{lcm}}

\def\N{\mathbb{N}}

\def\Z{\mathbb{Z}}
\def\Q{\mathbb{Q}}

\def\int{\mathrm{int}}
\renewcommand{\l}{{\tt L}}

\title{Union of sets of lengths of numerical semigroups}

\author{
J. I. Garc\'{\i}a-Garc\'{\i}a
 \footnote{
     Departamento de Matem\'aticas/INDESS (Instituto Universitario para el Desarrollo Social Sostenible),
     Universidad de C\'adiz, E-11510 Puerto Real  (C\'{a}diz, Spain).
     E-mail: ignacio.garcia@uca.es.}\\
D. Mar\'{\i}n-Arag\'on
 \footnote{
     Departamento de Matem\'aticas,
     Universidad de C\'adiz, E-11510 Puerto Real  (C\'{a}diz, Spain).
     E-mail: daniel.marin@uca.es.}
     \\
 A. Vigneron-Tenorio
 \footnote{
 	Departamento de Matem\'aticas/INDESS (Instituto Universitario para el Desarrollo Social Sostenible), Universidad de C\'adiz,
 	E-11406 Jerez de la Frontera (C\'{a}diz, Spain).
     E-mail: alberto.vigneron@uca.es.}
}

\date{\today}

\begin{document}

\maketitle

\begin{abstract}
Let $S=\langle a_1,\ldots,a_p\rangle$ be a numerical semigroup, $s\in S$ and $\z(s)$ its set of factorizations. The set of length is denoted by $\LL(s)=\{\L(x_1,\dots,x_p)\mid (x_1,\dots,x_p)\in\z(s)\}$ where $\L(x_1,\dots,x_p)=x_1+\ldots+x_p$. From these definitions, the following sets can be defined $\W(n)=\{s\in S\mid \exists x\in\z(s) \textrm{ such that  } \L(x)=n\}$, $\nu(n)=\cup_{s\in \W(n)} \LL(s)=\{l_1<l_2<\ldots< l_r\}$ and $\Delta\nu(n)=\{l_2-l_1,\ldots,l_r-l_{r-1}\}$. In this paper, we prove that the set  $\Delta\nu(S)=\cup_{n\in\N}\Delta\nu(n)$ is almost periodic with period $\lcm(a_1,a_p)$.
\end{abstract}

\section*{Introduction}
In many rings and semigroups, their elements can be written as a finite product (or sum) of other elements, but in general the factorizations are not unique, which is not what happens in the ring of integer numbers. The non-unique factorization theory describes and classifies these aspects using the invariants of the algebraic structure we are working with (see \cite{Chapman-Fontana-Geroldinger} for further background). From among these parameters we can highlight the $\omega$-primality, the tame-degree, the $\Delta$-set and the elasticity. What they try to measure, in a way or another, is how far is a semigroup or a ring from having unique factorizations, and if they are not unique they explain its behaviour. For example, if the $\Delta$-set of an element is the empty set, that means that the length of all its factorizations are the same. The computation of these parameters is always after a deep theoretical study because, in general, even if its definitions are not complicated, the computation for values not necessarily very high is not trivial and  it requires knowing some of its properties (bounds, periodicity, etc.) to be able to obtain effective algorithms for getting examples.

In recent years, a type of structure where these parameters have been well studied are the numerical and affine semigroups. We highlight, for example, the library "NumericalSgps" made in gap \cite{gapnumerical} where there are implemented functions to compute some of these parameters. Following this line, we can emphasize, the following works \cite{DeltaSets}, \cite{wprimalidad}, \cite{catenary} and many of the references therein.

In this work, we start from the definition of $\Delta$-set of the elements of a numerical semigroup and we define $\Delta$ of the union of a set of elements. Some papers where this parameter appears are the following: in \cite{Chapman-Smith} generalized sets of lengths are studied in Dedekind domains by Chapman and Smith and in \cite{Chapman-Smith-algebraic-number} its asymptotic behaviour is shown, in \cite{Amos-Chapman-Hine-Paixao} some properties of the set $\nu_n$ are obtained for numerical semigroups generated by an arithmetic progressions, in \cite{Baginski-Chapman-Hine-Paixao} the set $\Delta\nu(M)$ is computed for several monoids and the asymptotic behaviour of $\Delta\nu_n$ is also studied. This invariant has also been analyzed by Chapman, Freeze and Smith in \cite{Chapman-Freeze-Smith}. More recently, Geroldinger in \cite{set-of-lengths} made a survey for some parameters and proved using that $d=\min(\Delta(S))=\gcd\{a_{i+1}-a_i\mid i=1,\dots,p-1\}$, some results on the structure of $\nu_n$.  These sets are almost arithmetical progressions, and therefore  $\Delta\nu(S)\subset \{d,2d,3d,\dots\}$.

The main goal of this work is to give properties of the set of length of a numerical semigroup and to obtain algorithms who allow us to compute the function $\Delta\nu$. We prove that for its computation we do not need to calculate the $\Delta$-set of all the elements involved and therefore we  improve its computation in a remarkable way. We  also show that this function is almost periodic and we use this period and its bound for obtaining the function $\Delta\nu$ for any numerical semigroup. We provide some examples which illustate these algorithms. The software developed and all its examples can be downloaded in \cite{github}.

In Section \ref{seccion1} we give some basic definitions and introduce the notation that we use through this paper. Section \ref{seccion2} is devoted to explain the behaviour of the function $\Delta\nu$ and an improved algorithm for computing it is also given. Finally, in Section \ref{seccion3} we study the periodicity of $\Delta\nu$ and some examples are provided.

\section{Definitions and notations}\label{seccion1}

Denote by $\N$ the set on nonnegative integers. 
In this work $S$ denotes a primitive numerical monoid (or numerica semigroup). 
Since every numerical monoid is finitely generated, there exist $a_1,\dots, a_p\in\N$ such that $S=\langle a_1<\dots<a_p\rangle=\{\sum_{i=1}^p \lambda_i a_i \mid \lambda_1,\ldots ,\lambda_p\in \N\}$. 
If $M$ is the subgroup of $\Z^p$ defined by the equation $a_1x_1+\dots+a_px_p=0$ and $\sim_M$ is the equivalent relation on $\N^p$ defined by $z\sim_M z'$ if $z-z'\in M$, then the semigroup $S$ is isomorphic to the quotient $\N^p/\sim_M$.

Let $s$ be an element of $S$. If $(x_1,\dots,x_p)\in\N^p$ verifies that $\sum_{i=1}^p x_ia_i=s$, then we say that $(x_1,\dots,x_p)$ is a factorization of $s$. We denote by $\z(s)$ the set $\{(x_1,\dots,x_p)\in\N^p\mid a_1x_1+\dots+a_px_p=s\}$ and we call it the set of factorizations of $s$.

Define the linear function $\L:\Q^p\to \Q$ as $\L(x_1,\dots,x_p)=x_1+\dots+x_p$. 
The length of a factorization $x$ of $s\in S$ is the number $\L(x)$.

The following definition is found in {\cite{DeltaSets,DeltaNM-Chapman}}.

\begin{definition}
Given
$s\in S$ and $S=\langle a_1,\dots,a_p\rangle$, set
$
 \LL(s)=\{ \L(x_1,\dots,x_p)\mid (x_1,\dots,x_p)\in\z(s)
\}
$,
which is known as the set of lengths of $s$ in $S$.
Since $S$ is a numerical monoid, it is not hard to prove that
 this set of lengths is bounded, and so there exist  some positive integers
$l_1<\dots<l_k$ such that
$\LL(s)=\{l_1,\dots, l_k\}$. The set
\[
\Delta(s)=\{l_i-l_{i-1}: 2\leq i \leq k\}
\]
is known as the Delta set of $s$.\\

The set 
\[
\Delta(S)=\bigcup_{s\in S} \Delta(s)
\]
is called the Delta set of $S$.
\end{definition}

In \cite{DeltaSets}, it was proved that the function $\Delta:S\to \N$ is almost periodic.
The following definition is found in \cite{Amos-Chapman-Hine-Paixao, Baginski-Chapman-Hine-Paixao, set-of-lengths, Freeze-Geroldinger}.

\begin{definition}
	Let $S=\langle a_1,\dots,a_p\rangle$ and $n\in \N$.
	\begin{itemize}
		\item Define $\W(n)=\{s\in S\mid \exists x\in\z(s) \textrm{ such that  } \L(x)=n\}$.
		\item Define $\nu(n)=\cup_{s\in \W(n)} \LL(s)$
	\end{itemize}
	If $\nu(n)=\{l_1<l_2<l_3\dots<l_r\}$, then \[\Delta\nu(n)=\{l_2-l_1,l_3-l_2,\dots,l_r-l_{r-1}\}\]
	and \[\Delta\nu(S)=\cup_{n\in \N}\Delta\nu(n)\]
\end{definition}
Clearly, for every $n\in \N$ the set $\Delta\nu(n)$ is a subset of $\N$. Thus, for a $S$ numerical semigroup we define $\Delta\nu$ as follows:
\[
\begin{array}{rl}
\Delta\nu: & \N\to {\mathcal P}(\N)\\
 & n\to \Delta\nu(n)
\end{array}
\]

The main aim of this work is to prove that the above function is an almost periodic function and that its period is a divisor of $\lcm(a_1,a_p)$.

An unrefined method for computing $\Delta\nu(n)$ is the following:

\begin{algorithm}
\caption{Sketch of the algorithm to compute  $\Delta \nu(n)$.}\label{alg0}
\textbf{INPUT:} $S=\langle a_1,\dots,a_p \rangle$ a numerical semigroup and $n\in\N$.\\
\textbf{OUTPUT:} $\Delta\nu(n)$.
    \begin{algorithmic}[1]
        \State $A:=\{ (x_1,\dots,x_p)\mid \sum_{i=1}^p x_ia_i=n\}$.
        \State $\W(n):=\{ \sum_{i=1}^p x_ia_i\mid (x_1,\dots,x_p)\in A\}$.
        \State $\frak L=\cup_{s\in\W(n)}\LL(s)$.
        \State \Return $\Delta \frak L$.
    \end{algorithmic}
\end{algorithm}

The tuples $(n,0,0,\dots,0)$, $(n-1,1,0,\dots,0)$, \dots, $(0,n,0,\dots,0)$ are factorizations of different elements. So $\lim_{n \to + \infty} \#\W(n)=\infty$.

\begin{example}
    Let $S=\langle 5,9,11\rangle$ and $n=41$. The cardinality of $\W(41)$ is $123$ and for the computation of $\Delta \nu(41)$ using Algorithm \ref{alg0} it is necessary to know the factorizations of all of them. 
    In the following section, we prove that for any $n\in \N$ it is only necessary to calculate the factorizations of $111$ for computing $\Delta \nu (n)$.
    
    This number increase with $n$. For instance, if $n=50$ and $S=\langle 11,13,19\rangle$, this number is $255$, but Algorithm \ref{alg1} only need the computation of the factorizations of $111$ elements.
\end{example}

\section{Computation of $\Delta\nu(n)$}\label{seccion2}
In \cite{DeltaSets}, it is proved that there exists $\delta\in \N$ and a bound $N_S\in \N$ such that $\delta | lcm(a_1,a_p)$ and for every $s\in S$ with $s\geq N_S$ we have $\Delta(s+\delta)=\Delta(s)$.

It is straightforward to prove that $\min \W(n)=na_1$ y $\max\W(n)=na_p$. 
We use the notation of \cite{DeltaSets} and the elements $N_S$, $\vec w$, $\vec w'$ are defined as there.
We recall that explicitly these values are:
\[
d=\gcd\{a_{i+1}-a_i\mid i=1,\dots,p-1\},
\]
\[
S_i=
-\frac{a_2 \left(a_1 d \gcd \left(a_i-a_1,a_1-a_p,a_p-a_i\right)+(p-2) \left(a_1-a_i\right) \left(a_1-a_p\right)\right)}{\left(a_1-a_2\right) \gcd \left(a_i-a_1,a_1-a_p,a_p-a_i\right)},
\]
\[
S'_i=
\frac{a_{p-1} \left((p-2) \left(a_1-a_p\right) \left(a_p-a_i\right)-d a_p \gcd \left(a_i-a_1,a_1-a_p,a_p-a_i\right)\right)}{\left(a_{p-1}-a_p\right) \gcd \left(a_i-a_1,a_1-a_p,a_p-a_i\right)},
\]
\[
N_S=\lceil \max(\{S_i\mid i=2,\dots,p-1\}\cup\{S'_i\mid i=2,\dots,p-1\})\rceil,
\]
\[
\vec w=\frac{N_S(a_2-a_p)}{a_2(a_1-a_p)}e_1+\frac{N_S(a_1-a_2)}{a_2(a_1-a_p)}e_p-\frac{N_S}{a_1}e_1,
\]
\[
\vec w'=\frac{N_S(a_{p-1}-a_p)}{a_{p-1}(a_1-a_p)}e_1+\frac{N_S(a_1-a_{p-1})}{a_{p-1}(a_1-a_p)}e_p-\frac{N_S}{a_p}e_p.
\]

\begin{lemma}\label{lemaNprima}
	Let $S$ be a numerical semigroup and let $N_S$ the bound of \cite{DeltaSets}.
	There exists $N'_S\in \N$ such that for every $n\geq N'_S$ we have $\min \W(n)\geq N_S$. 
\end{lemma}
\begin{proof}
	The minimum of $\W(n)$ is equal to $n a_1$. It is enough to take $N'_S\geq \frac {N_S} {a_1}$.
\end{proof}

\begin{definition}{\cite[Definition 15]{DeltaSets}}\label{d:zetas}
	Let $S=\langle a_1,\dots, a_p\rangle$ be a numerical monoid.
	For every $s\in \N$ such that $s\geq N_S$, define
	\begin{itemize}
		\item $\z_1(s)$  the set of elements
		$x=(x_1,\dots,x_p)\in\z(s)$ verifying that $s/a_1+\l(\v w)< \l(x)\leq s/a_1$,
		\item $\z_2(s)$  the set of elements
		$x=(x_1,\dots,x_p)\in\z(s)$ verifying that $s/a_p+\l(\v w')-d\leq \l(x)\leq s/a_1+\l(\v w)+d$,
		\item $\z_3(s)$  the set of elements
		$x=(x_1,\dots,x_p)\in\z(s)$ verifying that $s/a_p\leq \l(x)< s/a_p+\l(\v w')$.
	\end{itemize}
\end{definition}

Note that $\L(\vec w)=\frac {(a_1-a_2)N_S}{a_1a_2}$ and
$\L(\vec w')=\frac {(a_p-a_{p-1})N_S}{a_pa_{p-1}}$.
Let $C_i$ be the following values:
\[C_1= \frac{\left(a_p-a_{p-1}\right) N_S}{a_{p-1}},~C_2=\frac{\left(a_1-a_2\right) N_S}{a_2},\]
\[C_3=\left(-\frac{a_p}{a_1}+\frac{a_p}{a_2}-\frac{a_p}{a_{p-1}}+1\right) N_S,
C_4=\left( \frac{a_1}{a_{p-1}}-\frac{a_1}{a_p}-\frac{a_1}{a_2}+1\right) N_S. \]

Define $\lambda_1=\max(C_1,C_4)$ and $\lambda_2=-\min(C_2,C_3)$.

\begin{proposition}\label{proplambda}
    For every $n\geq N_0=\max(\frac {N_S} {a_1},\frac{a_p-a_1+\lambda_1 + \lambda_2}{a_p-a_1})$ we have
    \[\Delta\nu(n)=\Delta ( \cup \{\z(x)| x \in [na_1,na_1+\lambda_1]\cup [na_p-\lambda_2,na_p]\} ).\]
\end{proposition}
\begin{proof}
Let $n\geq N_0$, by Lemma \ref{lemaNprima} we obtain that $x\geq N_S$ for all $x\in \W(n)$.\\

Using the properties of the sets $\z_i$ (Definition \ref{d:zetas}), for every $x\in \W(n)$ with $x\geq N_0$ there exists $c_1\in \z_1(x)$ such that $\L(c_1)=\min\{\L(x)\mid x\in \z_1(x)\}$ and $b_1\in \z_1(x)$ such that $\L(b_1)=\max\{\L(x) \mid x\in \z_1(x)\}$. We have that $\L(b_1)\leq \frac x {a_1}$ and that $\frac x {a_1} +\L(\vec w)\leq \L(c_1)$.
Analogously, 
there exists $c_2\in \z_3(x)$ such that $\L(c_2)=\min\{\L(x)\mid x\in \z_3(x)\}$ and $b_2\in \z_3(x)$ such that $\L(b_2)=\max\{\L(x) \mid x\in Z_3(x)\}$. Thus, $\frac x {a_p} \leq \L(c_2)$ and $\L(b_2) \leq \frac x {a_2} +\L(\vec w')$.\\

The following system of inequalities is obtained:
\begin{equation}\label{e1}
\frac x {a_p}>\frac {na_1} {a_p}+\L(\vec w'),
\end{equation}
\begin{equation}\label{e2}
\frac x {a_1}< \frac {na_p}{a_1}+\L(\vec w),
\end{equation}
\begin{equation}\label{e3}
\frac x {a_p} +\L(\vec w')<n+\L(\vec w),
\end{equation}
\begin{equation}\label{e4}
\frac x {a_1}+\L(\vec w)>n+\L(\vec w').
\end{equation}
These inequalities are summarized as follows:
\begin{equation}\label{ec5}
{na_1} + \lambda_1<x< {na_p}-\lambda_2. 
\end{equation}
If (\ref{e1}) y (\ref{e3}) are satisfied, then we get $\L(\z_1(x))\subset \L(\z_2(na_1))=\{d\}$. With (\ref{e2}) y (\ref{e4}), we obtain $\L(\z_3(x))\subset \L(\z_2(na_p))=\{d\}$.
From (\ref{e3}) y (\ref{e4}), we get $\L(\z_1(na_1))\subset \L(\z_2(x))$ y $\L(\z_3(na_p))\subset \L(\z_2(x))=\{d\}$. 
Finally, $\L(\z_1(x)\cup \z_3(x))\subset  \L(\z_2(na_1)\cup \z_2(na_p))=\{d\}$ and $\L(\z_1(na_1)\cup\z_3(na_p)\subset \L(\z_2(x))=\{d\}$.
Therefore, if there exists a solution of (\ref{ec5}), we obtain that $\Delta ( \cup \{\z(x)| x \in (na_1+\lambda_1,na_p-\lambda_2) \} )=\{d\}$.

To finish the proof, we now prove the existence of solutions of (\ref{ec5}).
Note that there exists $n$ such that $na_p-\lambda_2>na_1+\lambda_1$ and  $(na_p-\lambda_2)-(na_1+\lambda_1)>a_p-a_1$.
In this way there exists $k\in \N$ with $k\leq n$ such that  $na_1+\lambda_1<na_1+k(a_p-a_1)<na_p-\lambda_2$ and the element $na_1+k(a_p-a_1)$ belongs to $\W(n)$. 
This is fulfilled if $(na_p-\lambda_2)-(na_1+\lambda_1)>a_p-a_1$ which is satisfied if and only if 
\[n>\frac{a_p-a_1+\lambda_1 + \lambda_2}{a_p-a_1}.\]
Thus, we assert that there exists $x\in\W(n)$ satisfying (\ref{ec5}).
\end{proof}

\begin{figure}
\includegraphics[width=1.1\textwidth]{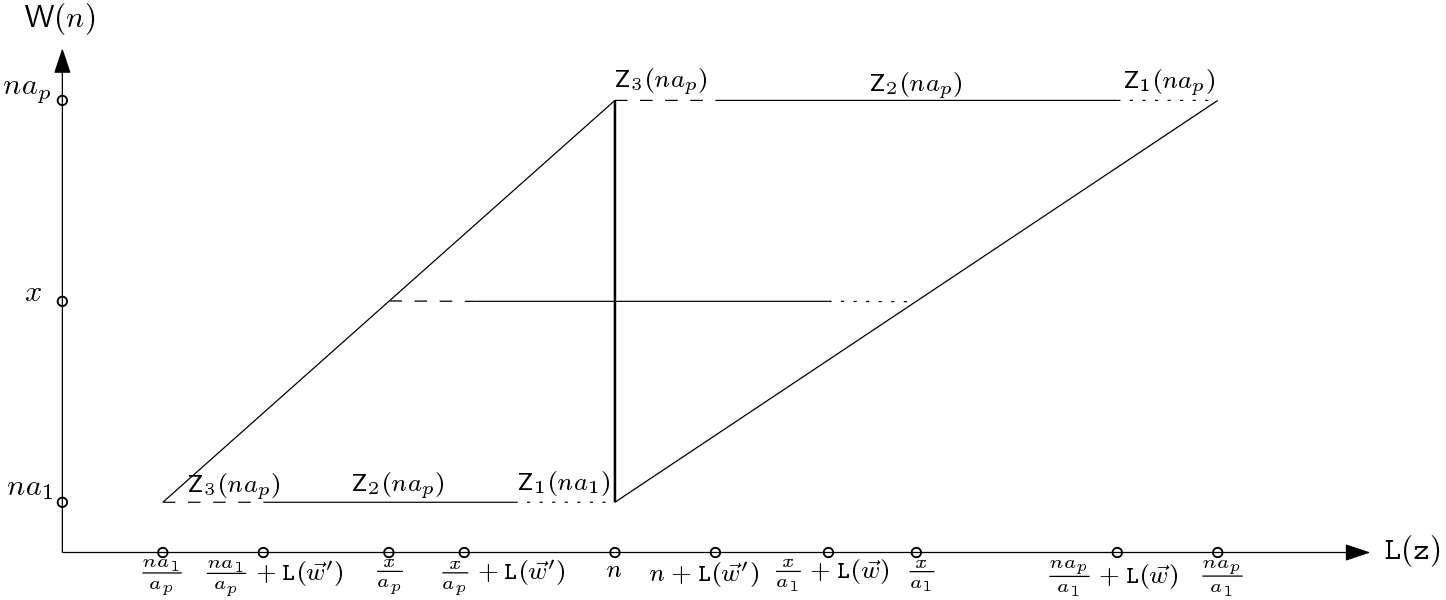}
\caption{Representation of the lenghts of the elements of $\W(n)$.}
\end{figure}

With the notation of the above proposition, we give the following definitions.

\begin{definition}
	Let $n\geq N_0$. Consider three zones in $\nu(n)$: $B_3(n)$, $B_2(n)$ y $B_1(n)$. Where
	$B_3(n)=\{x\in \nu(n) | x< \frac{na_1+\lambda_1} {a_p} \}$, 
	$B_1(n)=\{x\in \nu(n) | x> \frac{na_p-\lambda_2} {a_1} \}$  and 
	$B_2(n)=\nu(n)\setminus(B_1\cup B_3)$. 
\end{definition}

\begin{remark}\label{r1}
	From the construction given in \ref{proplambda}, We have that $\Delta\nu(n)=\Delta B_1(n)\cup \Delta B_2(n) \cup \Delta B_3(n)$ and $\Delta B_2(n)=\{d\}$. 
\end{remark}

\begin{algorithm}[H]
    \caption{Sketch of the algorithm to compute  $\Delta \nu(n)$.}\label{alg1}
    \textbf{INPUT:} $S=\langle a_1,\dots,a_p \rangle$ a numerical semigroup and $n\in\N$.\\
    \textbf{OUTPUT:} $\Delta\nu(n)$.
	\begin{algorithmic}[1]
	    \State $d:={\rm gdc}(a_2-a_1,\dots,a_p-a_{p-1})$.
	    \State Compute $N_S$ as in \cite[\S 3]{DeltaSets}.
	    \State $C_1 := \frac{\left(a_p-a_{p-1}\right) N_S}{a_{p-1}}$, $C_2 := \frac{\left(a_1-a_2\right) N_S}{a_2}$,\\
	    $C_3:=\left(-\frac{a_p}{a_1}+\frac{a_p}{a_2}-\frac{a_p}{a_{p-1}}+1\right) N_S$, $C_4:=\left( \frac{a_1}{a_{p-1}}-\frac{a_1}{a_p}-\frac{a_1}{a_2}+1\right) N_S$.
	    \State $\lambda_1:=\max(C_1,C_4)$, $\lambda_2:=-\min(C_2,C_3)$.
        \State Compute $N_0$ as in Proposition \ref{proplambda}.
	    \If {$n\leq N_0$}
	        \State Compute $\Delta \nu(n)$ using Algorithm \ref{alg0}.
	        \State \Return $\Delta \nu(n)$.
	    \EndIf

	    \State $x_1:=na_1+\lceil\lambda_1\rceil$.
	    \State $x_2:=na_p-\lfloor\lambda_2\rfloor$.
	    \State $W_3:=\W(n)\cap [na_1,x_1]$.
	    \State $B_3(n):=\{x\in \cup_{s\in W_3} \LL(s) \mid x\leq\frac{x_1}{a_p}\}$.
	    \State $W_1:=\W(n)\cap [x_2,na_p]$.
	    \State $B_1(n):=\{x\in \cup_{s\in W_1} \LL(s) \mid x\geq\frac{x_2}{a_1}\}$.
	    \State Compute $\Delta B_3(n)$.
	    \State Compute $\Delta B_1(n)$.
	    \State Return $\Delta B_3(n)\cup\{d\}\cup\Delta B_1(n)$.
	\end{algorithmic}
\end{algorithm}

\begin{example}
	Let $S$ be the numerical semigroup generated by $\langle 4,9,10,15 \rangle$. In this case $N_0=73$, this means that if we compute $\Delta\nu(n)$ with $n$ greater than it, for example $n=130$, we can save a lot of computations. In this case, $W(130)\subset[520,1950]$, $\lambda_1=203$, $\lambda_2=759$, $x_1=723$ and $x_2=1191$. Therefore, using Algorithm \ref{alg1} we have 468 values of $n$ that we can skip.
	
	The good part of this algorithm is that even if we increase the value of $n$, we only have to compute the same number of elements. For instance for $n=150$, $W(150)\subset [600,2250]$, but since $\lambda_1$ and $\lambda_2$ do not depend on $n$, we save $688$ evaluations.
\end{example}

\section{Periodicity of $\Delta\nu_{\_}:\N\to {\mathcal P}(\N)$}\label{seccion3}
The main result of this work is presented in this section. This result allows us to give some example where we compute the function $\Delta\nu$ for some numerical semigroups.

\begin{proposition}\label{periodo}
Let $n\geq N_0$. Then, 	$\Delta B_1(n)=\Delta B_1(n+ \mu a_1 )$, $\Delta B_3(n)=\Delta B_3(n+ \mu a_p )$ and $\Delta B_2(n)=\Delta B_2(n+ \mu a_i )$ for all $i\in\{1,\dots,p\}$ and for all $\mu\in \N$.

\end{proposition}
\begin{proof}
If $i=2$, then $\Delta B_2(n)=\Delta B_2(n+\mu a_i)=\{d\}$ for every $n\geq N_0$.

Take $i=3$. 
Let $x\in \Delta B_3(n)$. 
There exist $s_1,s_2\in [na_1,na_1+\lambda_1]\cap \W(n)$ and 
$z_1\in \z_3(s_1)$ and $z_2\in \z_3(s_2)$ fulfilling that $\L(z_1)-\L(z_2)=x$ there is no $z\in \nu(n)$ such that $\L(z_2)<\L(z)<\L(z_3)$.
Let $\tilde s_1=s_1+\mu a_p$ and $\tilde s_2=s_2+\mu a_p$. We have that $z_1+\mu e_p\in \z(\tilde s_1)$ and $z_2+\mu e_p\in \z(\tilde s_2)$ satisfying $\L(\tilde z_1)-\L(\tilde z_2)=x$. 
Furthermore, $\tilde s_1,\tilde s_2$ belong to $[na_1+\mu a_p,na_1+\lambda_1+\mu a_p]\cap \W(n+\mu a_p)$.

If there is an element $\tilde s\in \W(n+\mu a_p)$ with $\tilde z\in \z(\tilde s)$ such that $\L(\tilde z_2)<\L(\tilde z)<\L(\tilde z_3)$, when we consider the element $\tilde s-\mu a_p$ we obtain that such element has a factorization $z$ which verifies $\L(z_2)<\L(z)<\L(z_3)$ and this is a contradiction. So we have prove that $\Delta B_3(n)\subset \Delta B_3(n+\mu a_p)$. In the same way, the other inclusion can be proven so $\Delta B_3(n)= \Delta B_3(n+\mu a_p)$.

For $ i = 1 $, the demonstration is analogous.


\end{proof}

\begin{theorem}
Let $S$ be a numerical semigroup. The function $\Delta\nu :\N\to {\mathcal P}(\N)$ is almost periodic with period $\delta=\lcm(a_1,a_p)$. A bound from which this function is periodic is $N_0$.
\end{theorem}
\begin{proof}
From Proposition  \ref{periodo}, $\Delta B_2(n)=\{d\}$. On the other hand, $B_1$ and $B_3$ are periodics with period $a_3$ and $a_1$, respectively, so $\Delta B_1$ and $\Delta B_3$ has the same period. Now we use that $\Delta\nu(n)=\Delta B_1(n)\cup \Delta B_2(n)\cup \Delta B_3(n)$ in order to obtain $\Delta \nu$ has period $\lcm(a_1,a_p)$.
\end{proof}

Finally we illustrate the results of this work with some examples.
In these examples we show how we can compute $\Delta\nu(n)$ for several semigroups for all values of $n$. In order to get them, we have used a supercomputer \cite{super} checking the tree of numerical semigroups, in a parallel way, ordering the numerical semigroups by its genus and examining them. We discard the semigroups such that they are of the form $\langle m,m+k,\ldots,m+qk\rangle$ with $k,q\in\N$ since they are studied in \cite{Amos-Chapman-Hine-Paixao}.

\begin{example}
Here we have a collection of numerical semigroups with non-constant $\Delta\nu$.
\begin{itemize}
	    \item It i quite easy to find semigroups such that its $\Delta\nu$ has constant periodic part. For example, let $S$ be the semigroup $\langle 3,10,11\rangle$, we have that $N_0=82$, and $\delta=33$. Therefore we only have to compute the first $115$ values of $\Delta\nu$ in order to know all its values. After making this computations, we have the following results: $\Delta\nu(1)=\emptyset$, $\Delta\nu(2)=\Delta\nu(3)=\Delta\nu(4)=\Delta\nu(7)=\{1,2\}$ and $\Delta\nu(n)=\{1\}$ for  $n\in \{5,6\}\cup [8,33]$. So the real periodicity of this function is $1$, and because of this, if $n\geq 34$, $\Delta\nu(n)=\{1\}$. More semigroups having $\Delta \nu$ with this behaviour are: $\langle 10,13,15\rangle$, $\langle 4,7,9\rangle$ and $\langle 6,8,9,11 \rangle$.
	    \item A more interesting semigroup is the following one. If $S=\langle 3,10,14 \rangle$, we only need to compute $102$ values of $\Delta\nu$ since $N_0=60$ and $\delta=42$. The results are: \[\emptyset, \{1,4\}, \{1,3,4\}, \{1,3\}, \{1,3\}, \{1,4\}, \{1,2\}, \{1,3\}, \{1,4\}, \{1,2\},
	    \ldots.\] 
	    If $n\in [5,59]$, we have that $\Delta\nu(n)=\{1,4\}$ if $n\equiv 0\ \mod\ 3$, $\Delta\nu(n)=\{1,2\}$ if $n\equiv 1\ \mod\ 3$ and  $\Delta\nu(n)=\{1,3\}$ if $n\equiv 2\ \mod\ 3$.
	    If $n\geq 60 $, $\Delta\nu(n)=\{1,2\}$ if $n\equiv 0\ \mod\ 3$, $\Delta\nu(n)=\{1,3\}$ if $n\equiv 1\ \mod\ 3$ and  $\Delta\nu(n)=\{1,4\}$ if $n\equiv 2\ \mod\ 3$.
        The other values are $\Delta\nu(1)=\emptyset$, $\Delta\nu(2)=\{1,4\}$, $\Delta\nu(3)=\{1,3,4\}$,  and  $\Delta\nu(4)=\{1,3\}$. 
	    Hence, the real period is just 3.
        Another examples with with non-constant periodic part are 
	    $\langle 5,12,16 \rangle$, $\langle 6,13,17 \rangle$, $\langle 10,17,21 \rangle$, $\langle 17,24,28 \rangle$ and $\langle 4,9,10,15\rangle$.
	\end{itemize}
\end{example}

Thanks to our software (available in \cite{github}) its no difficult to obtain semigroups with non-constant $\Delta\nu$ and even with non-constant periodic part. This software has been developed in C++ for obtaining the maximum speed. However, in order to provide a friendly interface, we made an interface for Python3 and IPython3 (see \cite{python}) notebooks using swing (see \cite{swing}). Therefore, the user can load our library in a Jupyter notebook and use its Python functions which actually calls to our pre-compiled functions in C++, mixing the efficiency of C++ with the user-friendly Python.

\subsubsection*{Acknowledgements}
The authors were partially supported by Junta de Andaluc\'{\i}a research group FQM-366 and by the project MTM2017-84890-P (MINECO/FEDER, UE), and the third author is partially supported by the project MTM2015-65764-C3-1-P (MINECO/FEDER, UE).

\bibliographystyle{unsrt}
\bibliography{ref}


\end{document}